\documentclass[11pt]{article}

\usepackage{graphicx}
\usepackage{color}

\usepackage{amssymb, amsmath, amsthm}
\usepackage{float}

\newtheorem{dfn}{Definition}[section]
\newtheorem{defn}[dfn]{Definition}

\newtheorem{rem}[dfn]{Remark}

\newtheorem{thm}[dfn]{Theorem}
\newtheorem{lem}[dfn]{Lemma}
\newtheorem{lemma}[dfn]{Lemma}

\newtheorem{prop}[dfn]{Proposition}

\newtheorem{claim}[dfn]{Claim}
\newtheorem{cor}[dfn]{Corollary}

\begin{document}

\title{Strong Haken via Sphere Complexes}

\author{Sebastian Hensel and Jennifer Schultens}

\maketitle

\begin{abstract} 
	We give a short proof of Scharlemann's Strong Haken Theorem for closed $3$-manifolds (and manifolds with spherical boundary). As an application, we also show that
	given a decomposing sphere $R$ for a $3$-manifold $M$ that splits off an $S^2 \times S^1$ summand, any Heegaard splitting of $M$ restricts to the standard Heegaard splitting on the summand.
\end{abstract}
 \section{Introduction}
  Any (closed, oriented, connected) three-dimensional manifold $M$ admits a \emph{Heegaard splitting}, that is, it 
  can be decomposed into two three-dimensional handlebodies $V, V'$ of the same genus $g$ along 
  an embedded surface $S\subset M$:
	\[ M = V \cup_S V'. \]
  In theory, all information about the three-manifold is encoded in the identification of the two handlebodies. However, in practice, interpreting topological properties of $M$ using a Heegaard splitting is often nontrivial. 
  
  A basic example of this occurs when studying spheres in $M$. If $\alpha \subset S$ is a curve
  which bounds disks $D, D'$ in both $V$ and $V'$, then gluing these disks yields a $2$--sphere $D\cup D' \subset M$
  which intersects $S$ in the single curve $\alpha$. When essential, such a sphere is called a \emph{Haken sphere} -- but a priori it is completely unclear what kind of spheres in $M$ are of this form.

	A classical
  theorem of Haken shows that if $M$ admits any essential sphere $\sigma$, then it also admits a Haken sphere $\sigma'$. In fact, Scharlemann recently proved a \emph{Strong Haken Theorem}, showing that
  $\sigma'$ can in fact be chosen to be isotopic to $\sigma$:
 	\begin{thm} (Strong Haken Theorem) \label{GStrongHaken}
 		Let $M = V \cup_S V'$ be a Heegaard splitting.  Every essential $2$-sphere in $M$ is isotopic to a Haken sphere for $M = V \cup_S V'$.
	 \end{thm}
 The purpose of this article is to give an independent, short proof of Theorem~\ref{GStrongHaken} for any $M$ which is closed or has spherical boundary. 
 We want to mention that
 Scharlemann's version of the strong Haken theorem is in fact more general, allowing for manifolds with arbitrary boundary, and also showing that any properly embedded disk is isotopic to a Haken disk. This more general case 
 could also be obtained from our methods; for clarity we focus on the closed case throughout the article.
 
  To prove Theorem~\ref{GStrongHaken}, we make crucial use of the \emph{(surviving) sphere complex}, which is a combinatorial complex
  encoding the intersection pattern of essential (surviving) spheres in $M$. Such complexes have already been used successfully 
  in the study of outer automorphism groups of free groups (via mapping class groups of connected sums of
  $S^2\times S^1$). Here, we show that this perspective can also be useful in streamlining arguments in low-dimensional topology.
  The other crucial ingredient is the classical Waldhausen theorem on Heegaard splittings of the three-sphere.
  Together, these allow an inductive approach to Theorem~\ref{GStrongHaken}.
 	
  \medskip Our methods and results also allow to control Heegaard splittings of reducible manifolds.
	As a motivating example, we prove:
	\begin{prop} \label{HSofWn}
		Every Heegaard splitting of $W_n = n(S^2 \times S^1)$ is isotopic to a stabilization of the standard Heegaard splitting.  
	\end{prop}
  	
  	Combined the uniqueness for $W_1$ with the Strong Haken Theorem, we obtain the following structural result on Heegaard
  	splittings of arbitrary reducible three-manifolds.
  
  \begin{cor} \label{ReducibleSummand}
  	Given a reducible $3$-manifold $M$ with a Heegaard splitting $M = V \cup_S V'$, any decomposing sphere that splits off a $S^2 \times S^1$ summand can be isotoped so that $S$ is standard in this summand.  
  \end{cor}

\textbf{Acknowledgements} We would like to thank Martin Scharlemann for finding a mistake in an earlier
draft, and many helpful comments.

\section{Heegaard Splittings of Closed $3$--Manifolds}
In this section we recall some preliminaries on closed three-manifolds, their Heegaard splittings, and
spheres in such manifolds. The results presented here are classical.

\subsection{Heegaard Splittings}
\begin{defn} (Heegaard splitting) \label{HS}
	A {\em handlebody} is a $3$-manifold that is homeomorphic to a regular neighborhood of a graph in $S^3$.  A {\em Heegaard splitting} of a $3$-manifold $M$ is a decomposition $M = V \cup_S V'$, where $V, V'$ are handlebodies and $S = \partial V = \partial V' = V \cap V'$.
	
	The surface $S$ is called the {\em splitting surface} or {\em Heegaard surface}.  Heegaard splittings are considered {\em equivalent} if their splitting surfaces are isotopic.
\end{defn}

\begin{rem}
	The Heegaard splitting $M = V \cup_S V'$ is completely specified by the pair $(M, S)$, so we will sometimes write $(M, S)$ instead of $M = V \cup_S V'$.
\end{rem}

\begin{rem}
	The connected sum of two $3$-manifolds $M_1, M_2$ with Heegaard splittings $(M_1, S_1),$ $(M_2, S_2)$ inherits a Heegaard splitting $(M_1 \# M_2,$ $S_1 \# S_2)$.  This Heegaard splitting is unique in the sense that it is completely determined by the construction.  
	Later, we will briefly consider a refined notion of equivalence for Heegaard splittings where we distinguish between the two sides of the splitting surface.  With respect to this refined notion of equivalence, the Heegaard splitting of a connected sum is then not (a priori) unique, as the construction allows two different choices, namely which side of $S_1$ is identified to which side of $S_2$.
\end{rem}

\begin{defn}
	Given a Heegaard splitting $(M, S)$, the Heegaard splitting obtained from the pairwise connected sum $(M, S) \# (S^3, T)$, where $T$ is the standard unknotted torus in $S^3$, is called a {\em stabilization} of $(M, S)$.  A Heegaard splitting is {\em stabilized} if it is the stabilization of another Heegaard splitting and {\em unstabilized} otherwise. 
	
	A sphere that separates $(M, S) \# (S^3, T)$, {\it i.e.} a sphere that splits off a punctured $3$-ball containing an unknotted punctured torus, is called a {\em stabilizing} sphere.  
	
	A {\em stabilizing pair} of disks is a pair $(D, D')$ of disks such that $D$ is properly embedded in $V$, $D'$ is properly embedded in $V'$ and $\partial D \cap \partial D'$ is exactly one point. 
\end{defn}

\begin{rem}  
	A Heegaard splitting is stabilized if and only if it admits a stabilizing pair of disks.  Indeed, consider the standard unknotted torus in the $3$-sphere and observe that it separates $S^3$ into two solid tori.  The boundaries of the meridian disks of these solid tori intersect in exactly one point.
\end{rem}

A crucial theorem of Waldhausen's characterises all Heegaard splittings of the three-sphere.  See \cite{Waldhausen1968}.
\begin{thm}[Waldhausen's Theorem] \label{Waldhausen}  Every Heegaard splitting of the three-sphere is a stabilization of the unique standard genus $0$ Heegaard splitting.  \end{thm}

\subsection{Sphere Complexes}
A core tool in our argument is the following simplicial complex, which encodes the intersection patterns
of spheres in $M$.
\begin{defn}[Sphere complex]
	A sphere $S$ in a $3$-manifold is {\em compressible} if it bounds a $3$-ball.  Otherwise, it is {\em incompressible}. 
	We say that a sphere is \emph{peripheral} if it is isotopic into the boundary of the manifold.
	
	The {\em sphere complex} of a $3$-manifold $M$ is the simplicial complex ${\cal S}(M)$ determined by the following three conditions:
	
	\begin{enumerate}
		\item Vertices of ${\cal S}(M)$ correspond to isotopy classes of incompressible, non-peripheral embedded $2$-spheres;
		\item Edges of ${\cal S}(M)$ correspond to pairs of vertices with disjoint representatives;
		\item The complex ${\cal S}(M)$ is flag. 
	\end{enumerate}
\end{defn}
It is not hard to see that a simplex in the sphere complex corresponds to a collection of nonisotopic spheres
that can be realised disjointly. 

Furthermore, a standard argument involving surgery at innermost intersection circles shows that the
sphere complex of any closed $3$--manifold is connected (if it is nonempty). See
e.g. \cite{HatcherStability} for a proof in the case of doubled handlebodies, which also works in general.
\begin{figure}
	\includegraphics[scale=.5]{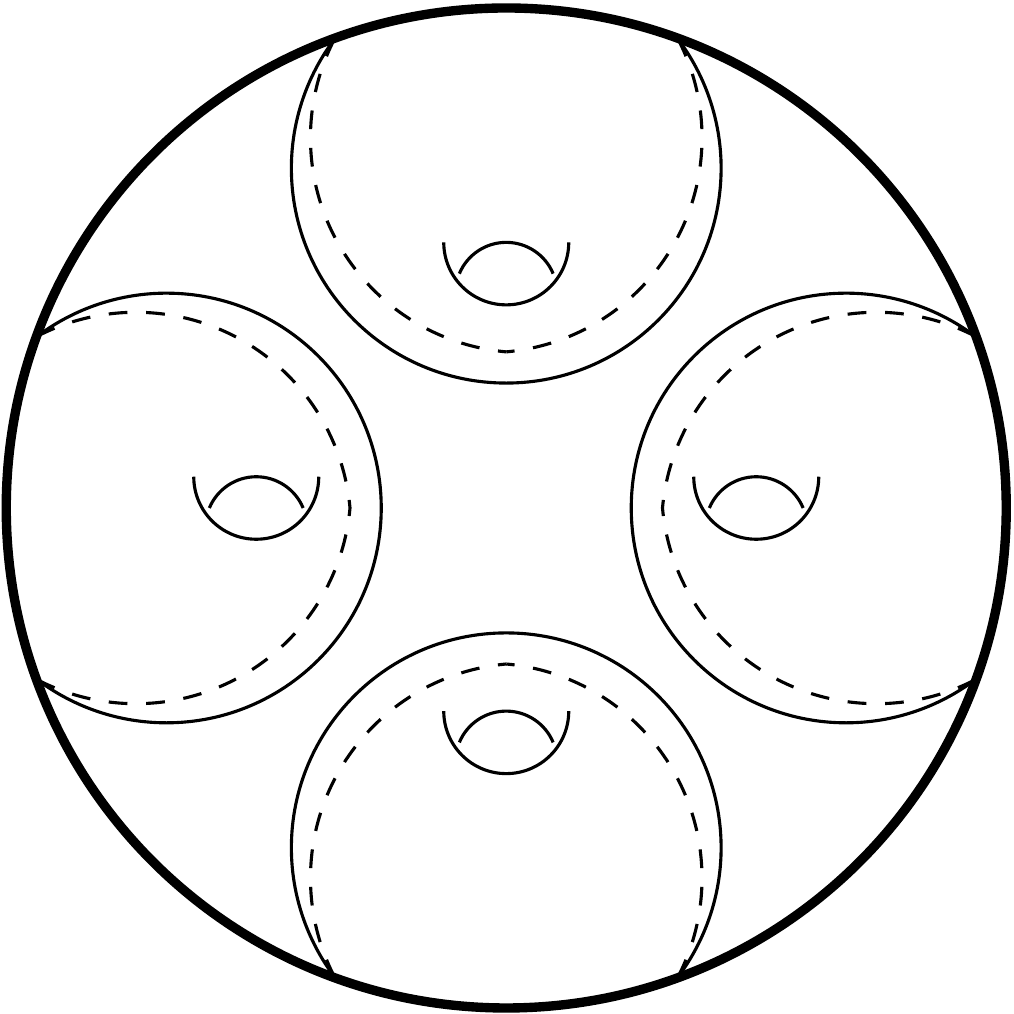}
	\centering
	\caption{\sl A $3$-simplex in the sphere complex of $W_4$, the double of a genus $4$ handlebody (alternatively, the connected sum of $4$ copies of $S^2 \times S^1$).}
	\label{standard}
\end{figure}

\subsection{Haken Spheres}
Our central aim will be to understand how essential spheres in $M$ interact with Heegaard splittings of $M$.
The following notion is crucial.
\begin{defn}
	Let $M = V \cup_S V'$ be a Heegaard splitting.  An essential sphere in $M$ that meets the Heegaard surface $S$ in a single simple closed curve is called a {\em Haken sphere}.  A (not necessarily essential) sphere that intersects $S$ in a single simple closed curve essential in $S$ is called a {\em reducing sphere}.
\end{defn}
The following theorem was originally proved by Haken in \cite{Haken}. Proofs can be found in the standard references on $3$-manifolds, see \cite{HempelBook}, \cite{JacoBook}, \cite{SchultensBook}.  

\begin{thm}[Haken's Lemma]\label{thm:hakens-lemma}
	If a closed three-manifold $M$ contains an essential sphere and $M = V \cup_S V'$ is a Heegaard splitting, then $M$ admits a Haken sphere. 
\end{thm}
In general, the Haken sphere is obtained by modifying the given essential sphere by surgery, and so cannot be guaranteed to
be related to the sphere given at the outset.

\section{Three-manifolds with spherical boundary}
In this section, we present versions of the results and notions of the previous section for $3$-manifolds with
spherical boundary. These appear naturally in our inductive proof of the Strong Haken Theorem (even if one is
just interested in proving it in the closed case).
For ease of notation, if $M$ is a $3$-manifold with spherical boundary, then we call each component of $\partial M$ a \emph{puncture}. Similarly, we call such a manifold a \emph{punctured manifold}.

\subsection{Heegaard Splittings}
To define Heegaard splittings of punctured manifolds, we use spotted handlebodies.
\begin{defn}
A {\em spotted} handlebody is a handlebody with a specified set of disks $D_1 \sqcup \dots \sqcup D_k$ in its boundary. Each disk is called a {\em spot}.  A Heegaard splitting of a $3$-manifold $M$ with spherical boundary is a decomposition $M = V \cup_S V'$ where $V, V'$ are spotted handlebodies with spots $D_1 \sqcup \dots \sqcup D_k$ and $D_1'  \sqcup \dots \sqcup D_k'$, respectively, and $S = \partial V - (D_1 \sqcup \dots \sqcup D_k) = \partial V' - (D_1' \sqcup \dots \sqcup D_k').$
\end{defn}

\begin{rem}
	In a Heegaard splitting of a $3$-manifold with spherical boundary each puncture meets the splitting surface in a single simple closed curve.  This simple closed curve is the boundary of a spot on each of the handlebodies.
	
	Suppose $M_1, M_2$ are two punctured manifolds, with boundary components $\partial_i \subset M_i$, and $M = M_1\cup_{\partial_1=\partial_2} M_2$ is the manifold obtained by gluing the boundary components.
	Given Heegaard splittings of $M_1, M_2$, the manifold $M$ inherits a Heegaard splitting which is obtained by
	gluing the handlebodies at the corresponding spots.
	
	The glued boundary components yield an essential $2$--sphere $\sigma$ in $M$, which intersects the induced
	Heegaard splitting in a single circle (i.e. it becomes a Haken sphere). Conversely, given a Haken sphere $\sigma$ for any manifold $M$, one can cut the manifold and the splitting at $\sigma$.
\end{rem}

We need a version of Waldhausen's Theorem in the context of punctured $3$-spheres (which is a fairly straightforward consequence of Waldhausen's theorem for $S^3$).
\begin{thm}[Waldhausen's Theorem for punctured $3$-spheres] \label{markedWaldhausen}  Every Heegaard splitting of a punctured $3$-sphere is a stabilization of a unique standard genus $0$ Heegaard splitting.  
\end{thm}
\begin{proof}
	Let $M$ be a punctured $3$-sphere and $M = V \cup_S V'$ a Heegaard splitting.  Construct $\hat M = S^3$ from $M$ by attaching a $3$-ball to each puncture.  By Alexander's Theorem, the result does not depend on the attaching map.  Moreover, the attaching maps can be chosen so that a meridional disk of each $3$-ball caps off a component of $\partial S$.  We thus obtain a closed surface $\hat S$ that defines a Heegaard splitting $S^3 = \hat V \cup_{\hat S} \hat V'$.  
	
	Each $3$-ball that has been attached to a puncture is a regular neighborhood of a point and, as such, arbitrarily small.  By Waldhausen's Theorem, $S^3 = \hat V \cup_{\hat S} \hat V'$ is a stabilization of the standard genus $0$ Heegaard splitting of $S^3$.  The stabilizing pairs of disks can be chosen to be disjoint from the attached $3$-balls.  Thus, after destabilizing, if necessary, we may assume that $S$ is genus $0$.  
	
	Hence, to show the theorem, it suffices to show that any genus $0$ splitting of a punctured $S^3$ is standard.
	To this end, observe that the spotted genus $0$ handlebody
	$V\subset S^3$ can be isotoped to be a regular neighbourhood of a graph $\Gamma \subset V$. The graph $\Gamma$ can be chosen to have the following form: it has one vertex $v_0$ in the interior of $M$, and one vertex on each boundary component. Each vertex on a boundary component is joined to $v_0$ by an edge.
	Now, observe that any two such graphs are isotopic, as any two arcs from $v_0$ to a boundary sphere are isotopic, by the Lightbulb Trick.
	This shows that any two genus $0$ splittings of a punctured $S^3$ are isotopic.
\end{proof}

At this point, we briefly want to address the ambiguity appearing in the previous proof when filling the boundary and drilling it out again -- namely, one can isotope a pair of stabilizing disks across a puncture. This leads to a homeomorphism of the manifold preserving the Heegaard surface.
Given a Heegaard splitting of a 3-manifold, the \emph{Goeritz group} of the splitting is the group of isotopy classes of orientation preserving diffeomorphisms of the manifold that preserve the splitting.  Loosely speaking, the Goeritz group will be small if the surface automorphism that defines the Heegaard splitting is complicated relative to the handlebodies.  Conversely, the Goeritz group will be as large as possible in the case of $W_n$, the manifold for which this surface automorphism is the identity, and the Goeritz group is equal to the handlebody group.
Scharlemann finds a system of $4g+1$ generators for the Goeritz group of a handlebody (see \cite{ScharlG}).
On the other hand, the Goeritz group of the three-sphere is still largely mysterious. We refer the interested reader to the recent \cite{ScharlemannGoeritz} and there references therein.

\subsection{Sphere Complexes}
We now want to define a useful sphere complex for punctured manifolds. One obvious change is that for the vertices
one should also exclude \emph{peripheral}
spheres, i.e. spheres which are homotopic into the boundary (otherwise, such spheres are adjacent to any other vertex, rendering the resulting complex useless). However, even with this modification, the resulting sphere complex will be somewhat problematic for our purposes, as it may often be disconnected. Namely, suppose that $M_0$ is an aspherical three-manifold with infinite fundamental group. Let $M$ be the manifold obtained from $M_0$ by removing two open balls. The manifold $M$ admits many essential non peripheral spheres obtained by joining the two punctures by a nontrivial tube. In fact, by asphericity of $M_0$, any essential non peripheral sphere in $M$ is of this form. In particular, no two such are disjoint. 

\smallskip To sidestep this issue, we use the following variant of the sphere complex. 
\begin{defn}[Surviving Sphere complex]
	A sphere $S$ in a punctured $3$-manifold is \emph{almost peripheral} if it bounds a punctured $S^3$ in $M$.  If $S$ is not almost peripheral, then it is {\em surviving}.  
	
	The {\em surviving sphere complex} of a $3$-manifold $M$ is the simplicial complex ${\cal S}^s(M)$ determined by the following three conditions:
	\begin{enumerate}
		\item Vertices of ${\cal S}^s(M)$ correspond to isotopy classes of incompressible, embedded, surviving $2$-spheres;
		\item Edges of ${\cal S}^s(M)$ correspond to pairs of vertices with disjoint representatives;
		\item The complex ${\cal S}^s(M)$ is flag. 
	\end{enumerate}
\end{defn}
The terminology stems from the fact that the spheres ``survive filling in the punctures'' and is in analogy to the surviving curve complex used in the study of mapping class groups of surfaces, see {\em e.g.}, \cite{BHW, GLP}. 
It turns out that these complexes are much better behaved in our setting.
\begin{lem}\label{lem:connectivity-sphere-graph}
	Let $M$ be a $3$--manifold. Then the surviving sphere complex ${\cal S}^s(M)$ is connected (if it is nonempty).
\end{lem}
\begin{proof}
	Let $\sigma, \sigma'$ be two incompressible, embedded, surviving $2$-spheres in $M$. Up to isotopy, we may
	assume that $\sigma, \sigma'$ intersect transversely. Further, we may assume that up to isotopy, the number
	of intersection components $\sigma \cap \sigma'$ is minimal.
	
	Let $C \subset \sigma\cap \sigma'$ be an innermost
	intersection circle, i.e. suppose that it bounds a disk $D \subset \sigma$ with $D \cap \sigma' = \partial D = C$.
	Denote by $S^+, S^- \subset \sigma'$ the two disks bounded by $C$, and denote by $\sigma^\pm = S^\pm \cup D$ the
	two $2$--spheres obtained by disk-swapping. Observe that up to isotopy, both of these are disjoint from $\sigma'$, and intersect $\sigma$ 
	in at least one fewer circle than $\sigma'$.  If either $\sigma^+, \sigma^-$ were compressible, then we could reduce the number of components in $\sigma \cap \sigma'$ by sliding $\sigma$ over the ball bounded by the compressible sphere, which is impossible by our choice.
	
	Assume that $\sigma^-$ is almost peripheral. Then, after filling in the punctures of $M$, the spheres $\sigma$ and
	$\sigma^+$ are isotopic (by sliding $D$ over the now unpunctured ball bounded by $\sigma^-$). In particular, $\sigma^+$ is surviving, as the same is true for $\sigma$.
	
	Hence, at least one of $\sigma^{\pm}$ is surviving, and we are done.  Indeed, repeating this process produces a sequence of spheres corresponding to vertices in a path, in ${\cal S}^s(M)$, between $[\sigma]$ and $[\sigma']$. 
	\begin{figure}
		\includegraphics[width=.5\textwidth]{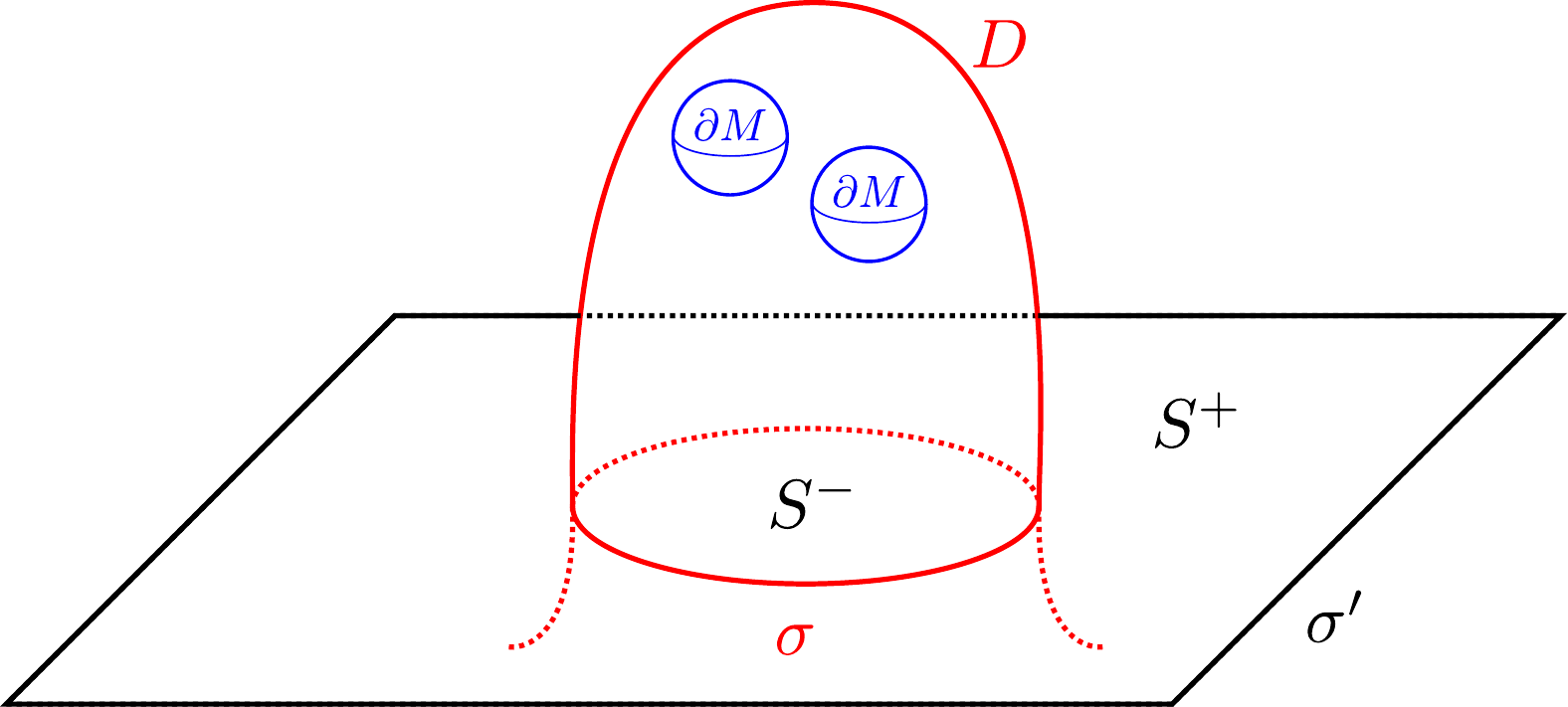}
		\centering
		\caption{\sl In the proof of Lemma~\ref{lem:connectivity-sphere-graph}: the innermost intersection circle
			of $\sigma, \sigma'$ cuts $\sigma'$ into two disks $S^+, S^-$. 
			If $\sigma^-$ is almost peripheral, then $\sigma^+$ is isotopic to $\sigma$ after filling the punctures.}
		\label{fig:nonperipheral1}. 
	\end{figure}
\end{proof}

\subsection{Haken Spheres}
Just as in the closed case, we call an essential sphere which intersects a Heegaard splitting of a punctured
manifold in a single curve a \emph{Haken sphere}. For punctured manifolds, almost peripheral and surviving Haken spheres behave slightly differently. 

On the one hand, using the same strategy as in the proof of Theorem~\ref{markedWaldhausen}, we obtain the 
following corollary of Theorem~\ref{thm:hakens-lemma}.
\begin{thm}[Surviving Haken's Lemma]\label{thm:nonpreperipheral-hakens-lemma}
If $M$ contains an surviving sphere and $M = V \cup_S V'$ is a Heegaard splitting, then there is a surviving Haken sphere.
\end{thm}
\begin{proof}
	Denote by $\sigma$ an essential surviving sphere in $M$.
	Let $M'$ be the three-manifold obtained from $M$ by gluing a ball to each boundary component. Denote by
	$B\subset M'$ the disjoint union of the resulting balls. By definition
	of almost peripheral, the image of $\sigma$ in $M'$ is still essential. Thus, Haken's lemma 
	(Theorem~\ref{thm:hakens-lemma}) applies, and yields a Haken sphere $\sigma' \subset M'$. By an isotopy
	preserving the Heegaard surface, we may assume that $\sigma'$ is disjoint from $B$. We can thus interpret
	$\sigma'$ as a sphere in $M\subset M'$, where it is the desired Haken sphere.
\end{proof}

On the other hand, almost peripheral spheres are also Haken spheres:
\begin{lemma}[Almost Peripheral Strong Haken Theorem]\label{lem:preperipheral-strong-haken}
	Let $M$ be a three-manifold with at least two punctures, and $M = V \cup_S V'$ be a Heegaard splitting.
	Then any almost peripheral sphere $\sigma$ in $M$ is isotopic to a Haken sphere.
\end{lemma}
\begin{proof}
	We begin with the case where $\sigma$ cuts off exactly two punctures $\delta_1, \delta_2$. The almost peripheral sphere
	$\sigma$ is then isotopic to the boundary of a regular neighbourhood of $\delta_1 \cup \alpha \cup \delta_2$, where 
	$\alpha \subset M$ is a properly embedded arc. We may homotope $\alpha$ to lie in $S$, as any arc in a handlebody
	is homotopic into the boundary. However, the arc may now not be embedded anymore. We can remove the self-intersections by ``popping subarcs over $\delta_1\cap S$''. 
	\begin{figure}
		\includegraphics[width=.5\textwidth]{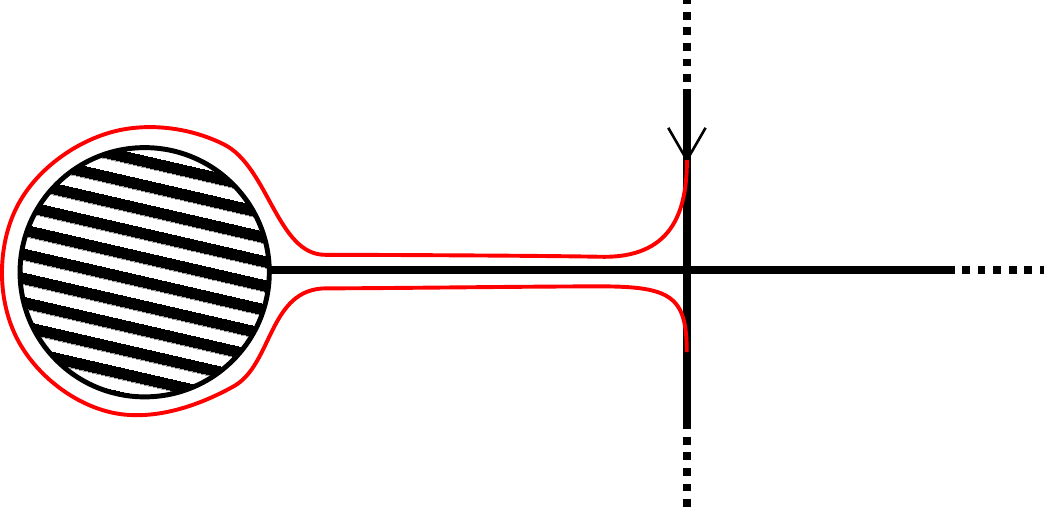}
		\centering
		\caption{\sl Removing self-intersections of an arc joining two spots.}
		\label{fig:resolve-intersection}. 
	\end{figure}
	To be more precise, parametrise $\alpha:[0,1]\to S$ so that it starts on $\delta_1\cap S$, and homotope so that all self-intersections are transverse.
	Consider the first self-intersection point $\alpha(t)=\alpha(s), t < s$. In particular, this implies that
	$\alpha\vert_{[0,t]}$ is an embedded arc. 
	
	Now homotope a small subarc $\alpha\vert_{[s-\epsilon,s+\epsilon]}$ to instead be the arc obtained by following 
	$\alpha\vert_{[0,t]}$ backwards to $\delta_1\cap S$, following around $\delta_1\cap S$, and returning along
	$\alpha\vert_{[0,t]}$ (compare Figure~\ref{fig:resolve-intersection}). This homotopy is possible in $V$, and the resulting arc has at least one fewer self-intersection.
	
	After a finite number of modifications of this type, the boundary of a regular neighbourhood of $\delta_1\cup\alpha\cup\delta_2$ (which is 
	homotopic to $\sigma$) interects $S$ in a single curve, and thus is a Haken sphere. By a theorem of Laudenbach, see \cite{Laudenbach73}, homotopy and isotopy are the same for spheres in $3$--manifolds, hence the claim follows\footnote{One could also avoid citing this theorem by isotoping $\alpha$ into a regular neighbourhood of $S$ and resolving crossings of the projection to $S$ by isotopies which slide strands over the puncture similar to Figure~\ref{fig:resolve-intersection}.}.
	
	Now we suppose $\sigma$ is a sphere cutting off $k>2$ punctures. Then there is a sphere $\sigma'$, disjoint from $\sigma$, which cuts off $2$ punctures, and which is contained in the punctured $S^3$ bounded by $\sigma$. By the initial case, $\sigma'$ may be assumed to be Haken. Let $M'$ be the manifold obtained by cutting $M$ at $\sigma'$, with the induced Heegaard splitting; 
observe that $\sigma \subset M'$ is still almost peripheral, but now cuts off at most $k-1$ spheres. By induction, $\sigma$ is a Haken sphere.
\end{proof}
Since any essential, non peripheral sphere in a punctured $S^3$ is almost peripheral, this implies the following:
\begin{cor}[Strong Haken Theorem for punctured $3$--spheres]\label{cor:strong-haken-s3}
	Any essential sphere in a punctured $3$-spheres is isotopic to a Haken sphere.
\end{cor}

\section{Heegaard splittings of $n(S^1 \times S^2)$} \label{HSWn}

In this section, we study Heegaard splittings of a specific manifold, namely
\begin{defn}
	We denote the double of the genus $n$ handlebody by $W_n$.  It is the connected sum of $n$ copies of $S^2 \times S^1$. 
\end{defn}
A reader only interested in the strong Haken theorem may safely skip ahead to the next section.
Our goal here will be to prove that, similar to Waldhausen's theorem for the three-sphere, all Heegaard
splittings of $W_n$ are ``standard'' in the following sense. 

\begin{defn}
	A Heegaard splitting of $W_n$ is {\em standard} if it is the double of a genus $n$ handlebody.  A {\em standard} Heegaard splitting of $W_n$ is a Heegaard splitting that is the connected sum of $n$ copies of $W_1$ with the standard Heegaard splitting.
\end{defn}

Waldhausen seems to claim in \cite{Waldhausen1968} that all Heegaard splittings of $W_n$ are standard (although
it is not entirely clear up to which equivalence relation, and the proof sketch is incomplete). In the unpublished preprint \cite{Oertel}, Oertel and Navarro Carvalho prove the result, using results on the homeomorphism groups of handlebodies and $W_n$ (in a very similar way to the argument we will use below). In this section, we show that these techniques could also be used to prove a Strong Haken theorem (and obtain the uniqueness of splittings as a corollary). We want to emphasize that this of course follows from the 
general Strong Haken theorem (Theorem~\ref{GStrongHaken}), but consider the argument using homeomorphisms
of $W_n$ interesting enough to warrant this alternate proof.
\smallskip

\begin{prop} \label{genus1}
Every unstabilized Heegaard splitting of $W_1$ is standard.  
\end{prop}

\proof
Suppose that $W_1 = V \cup_S V'$ is a Heegaard splitting.  We wish to show that $W_1 = V \cup_S V'$ is standard.  Since $W_1$ is reducible, Haken's lemma tells us that there is a Haken sphere $R$ for $W_1 = V \cup_S V'$.  Denote $V \cap R$ by $D$ and $V' \cap R$ by $D'$.  Note that all essential spheres in $W_1$, in particular $R$, are isotopic to $S^2 \times (point)$. 

We may assume that $S$ intersects a bicollar of $R$ in an annulus $(S \cap R) \times [-1,1]$.  Removing this bicollar leaves a copy of $S^2 \times [-1,1]$, {\it i.e.,} a twice punctured $3$-sphere that inherits a Heegaard splitting.  By Theorem \ref{markedWaldhausen}, this Heegaard splitting is either of genus $0$ or stabilized.

Since $W_1 = V \cup_S V'$ is unstabilized, the Heegaard splitting obtained on the complement of $S^2 \times [-1,1]$ must be of genus $0$.  Specifically, the splitting surface is a twice punctured $2$-sphere, {\it i.e.,} an annulus.  Hence we can reconstruct $W_1 = V \cup_S V'$:  Indeed, say, $V$, is composed of a $3$-ball attached to the two copies $D \times \{\pm 1\}$ of $D$.  It follows that $V$ is a solid torus.  
The same is true of $V'$ and hence $W_1 = V \cup_S V'$ is the standard Heegaard splitting. 
\qed

First, we have
the following classical result due to  Griffiths \cite{Griffiths}.
\begin{thm}\label{thm:griffiths}
	The action of the mapping class group of a handlebody $V_n$ on its fundamental group $\pi_1(V_n) = F_n$ induces
	a surjection
	\[ \mathrm{Mcg}(V_n) \to \mathrm{Out}(F_n) \to 1. \]
\end{thm}
We remark that the kernel of this map is quite complicated, and generated by twists about disk-bounding
curves \cite{Luft}.
Next, we need a theorem of Fran\c{c}ois Laudenbach
\cite{Laudenbach73} (see also \cite{Laudenbach74} and \cite{BrendleBroaddusPutman} for a modern proof):
\begin{thm}\label{thm:laudenbach}
	The action of the mapping class group of a doubled handlebody $W_n$ on its fundamental group $\pi_1(W_n) = F_n$ induces
	a short exact sequence
	\[ 1 \to K \to \mathrm{Mcg}(W_n) \to \mathrm{Out}(F_n) \to 1. \]
	The kernel $K$ is finite, generated by Dehn twists about non-separating spheres, and acts trivially 
	on the isotopy class of every embedded sphere or loop.
\end{thm}
\begin{cor}\label{cor:strong-haken-standard}
	For the standard Heegaard splitting of $W_n$, every essential sphere in $W_n$ is isotopic to a Haken sphere.
\end{cor}
\begin{proof}
	First observe that any two nonseparating spheres in $W_n$ can be mapped to each other by a homeomorphism. 
	Namely, the complement of such a sphere is homeomorphic to $W_{n-1}$ with two punctures.
	Similarly, separating spheres can be mapped to each other if and only if the fundamental groups of the complements are free groups of the same rank (as the complement is a disjoint union of once-punctured $W_k$ and $W_{n-k}$).
	
	Next, observe that there are Haken spheres of all such possible types, obtained by doubling a suitable disk
	in the handlebody. 
	
	Let $i:V_n \to W_n$ be the inclusion induced by doubling. Observe that on the one hand, the boundary of $V_n$ 
	maps under $i$ to the standard Heegaard splitting of $W_n$, and on the other hand $i$ induces an isomorphism $i_*$
	of fundamental groups. For any  outer automorphism $\varphi\in \mathrm{Out}(\pi_1(V_n))$ of the fundamental group of $V_n$, by Theorem~\ref{thm:griffiths}, there is a homeomorphism $f: V_n \to V_n$ inducing it. 
	Let $F:W_n \to W_n$ be the homeomorphism of $W_n$ obtained by doubling $f$. Observe that $F$ preserves the standard Heegaard splitting of $W_n$ by construction, and $F$ induces $\varphi$ via the isomorphism
	$i_*:\pi_1(V_n) \to \pi_1(W_n)$. Since $\varphi$ was arbitrary, this shows that
	any outer automorphism of $\pi_1(W_n)$ can in fact be realised by a homeomorphism
	of $W_n$ preserving the standard Heegaard splitting. 
	
	Together with Laudenbach's Theorem~\ref{thm:laudenbach} this shows that any sphere is isotopic to the
	image of a Haken sphere under a homeomorphism preserving the standard Heegaard splitting -- hence, it is
	isotopic to a Haken sphere.
\end{proof}

\begin{lem} \label{genusn}
There is a unique Heegaard splitting of $W_n$ of genus $n$.
\end{lem}

\proof 
Connected sum decompositions of $W_n$ are not unique.  However, let $W_n = V \cup_S V'$ be the standard Heegaard splitting and let $W_n = X \cup_Y X'$ be any Heegaard splitting of genus $n$.   By repeated application of Theorem \ref{thm:hakens-lemma} there are Haken spheres $R_1 \cup \dots \cup R_{n-1}$ for $W_n$ that cut $W_n = X \cup_Y X'$ into standard Heegaard splittings of $W_1$.  By Corollary \ref{cor:strong-haken-standard}, $R_1, \dots, R_{n-1}$ are also Haken spheres for $W_n = V \cup_S V'$.  By an Euler characteristic argument, these cut $W_n = V \cup_S V'$ into genus $1$ Heegaard splittings of the summands.  By Proposition \ref{genus1} these are standard.  In particular, $Y$ is isotopic to $S$.  
\qed

\begin{proof}[Proof of Proposition \ref{HSofWn}]
For an unstabilized Heegaard splitting, this is Lemma \ref{genusn}.  Furthermore, if $n = 1$, then this follows from Proposition \ref{genus1}. So suppose that $n > 1$ and $W_n = V \cup_S V'$ is stabilized.  By Corollary \ref{cor:strong-haken-standard}, there is a Haken sphere $R$ that decomposes $W_n$ into $W_1 \# W_{n-1}$.
Moreover, by \cite{QiuScharlemann}, one of the Heegaard splittings inherited by the summands is stabilized.  By induction, $W_n = V \cup_S V'$ is a stabilization of a connected sum of standard Heegaard splittings, {\it i.e.,} a stabilization of the standard Heegaard splitting of $W_n$.  
\end{proof}

We finish this section with a quick discussion on ``flippable'' Heegaard splittings. Namely, so far we have
considered Heegaard splittings up to isotopy of the splitting surface.
One could define a more refined ``oriented'' version of equivalence by requiring that the
handlebodies are isotopic (as opposed to just their boundary surface).

We want to remark that for the manifolds $W_n$ we have been considering, the two notions of equivalence
are in fact equal. Given the uniqueness shown above, to see this we just have to observe that the standard
splittings of $W_n$ are ``flippable'' -- i.e. there is an isotopy exchanging the ``inside'' and ``outside'' handlebodies.

For $W_1$ this can be seen by hand (by rotating the $S^2$ factor of $S^1\times S^2)$. Furthermore, by considering a separating Haken sphere, we see that the connected sum of flippable Heegaard splittings is flippable and, similarly, that stabilizations of flippable Heegaard splittings are flippable.  By Theorem \ref{HSofWn}, all Heegaard splittings of $W_n$ are stabilizations of the standard Heegaard splitting.  Hence all Heegaard splittings of $W_n$ are flippable. 

\begin{figure}
\includegraphics[scale=.6]{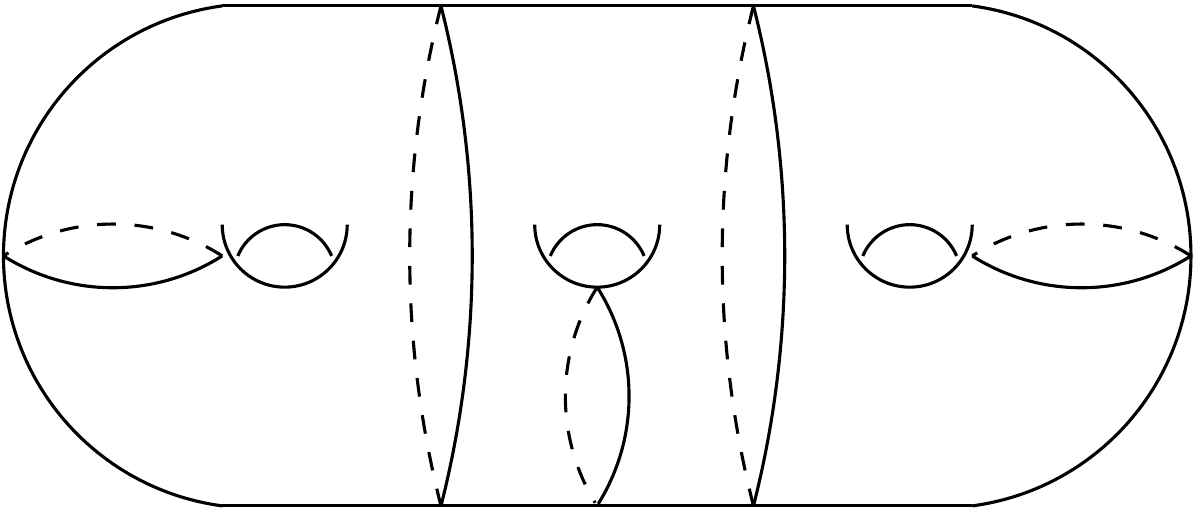}
\centering
\caption{\sl A doubled handlebody can be flipped}
\label{handlebody}
\end{figure}

\section{Strong Haken Theorem}

Combining the uniqueness of Heegaard splittings for $W_n$ (Proposition~\ref{HSofWn}) with Corollary~\ref{cor:strong-haken-standard} yields a \emph{Strong Haken Theorem} for $W_n$: any sphere
in $W_n$ is isotopic to a Haken sphere.  This statement was recently proved by Scharlemann \cite{Scharlemann2020} for all three-manifolds. In this section, we provide a short independent proof of this theorem for closed manifolds and manifold with spherical boundary.

The following proof proceeds by two nested inductions.  It naturally involves $3$-manifolds with spherical boundary, even if we just want to prove the theorem in the closed case. Recall that for such $3$-manifolds, we decree that each boundary sphere (puncture) meets the splitting surface in a single simple closed curve.

\begin{proof}[{Proof of Theorem~\ref{GStrongHaken}}]
	We prove the theorem by considering all punctured $3$-manifolds and all Heegaard splittings, ordered
	according to a suitable complexity. Namely, if $M = V \cup_S V'$ is a Heegaard splitting, we define
	the \emph{complexity} as the pair $(g(S), n(S))$ of genus and number of spots of the handlebodies
	(ordered lexicographically). We perform a nested induction on the genus $g$ and the number of boundary components $n$. The argument for the inductive step is in fact the same in both cases, and so we describe both inductions 
	simultaneously.
	
	\paragraph{Induction Start $g=0, n\geq 0$} Observe that the only punctured $3$--manifold that can be 
	obtained from a Heegaard splitting of genus $0$ is the three-sphere. Thus, the Strong Haken Theorem in this case is simply Corollary~\ref{cor:strong-haken-s3}.
	
	\paragraph{Induction Steps} Now suppose that the Strong Haken Theorem is known for all manifolds of complexity at most $(g,n)$, and suppose that $M$ is a manifold of complexity $(g, n+1)$, or suppose that the Strong Haken Theorem is known for all manifolds of complexity $(g, k), k \geq 0$, and $M$ is a manifold of complexity $(g+1, 0)$.
	
	First observe that, by Lemma~\ref{lem:preperipheral-strong-haken} any almost peripheral sphere in $M$ is isotopic
	to a Haken sphere. We thus have to show that surviving spheres in $M$ are also isotopic to Haken spheres.
	We now make the following
	\begin{claim}
		Suppose that $R$ is a surviving Haken sphere in $M$, and suppose that $R'$ is a surviving sphere disjoint from $R$. Then $R'$ is a Haken sphere.
	\end{claim}
	\begin{proof}[Proof of Claim]\renewcommand{\qedsymbol}{//}
		Denote by $M-R$ the punctured $3$--manifold obtained by cutting at $R$. $M-R$ has two punctures more than $M$, corresponding to the two sides of $R$. $M-R$ has one or two components, depending on whether $R$ is separating or not.
		
		Let $M'$ be the component of $M-R$ containing $R'$. This manifold inherits a Heegaard splitting from $V \cup_S V'$ with splitting surface a component of $S' = S - (R\cap S)$.
		If $R\cap S$ is nonseparating, then $g(S') < g(S)$. If $R \cap S$ is separating, then either the genus
		or the number of boundary components is smaller for $S'$. In either case, $(g(S'),n(S'))<(g(S),n(S))$ lexicographically (crucial here is the fact that $R'$ is an essential sphere in $M'$: it cannot become compressible, as that would violate incompressibility in $M$, therefore $S - S'$ is not a disk; furthermore, it cannot become peripheral as it would then be peripheral in $M$, or isotopic to $R$, so $S - S'$ is not an annulus.) 

		Now, if $R'$ is almost peripheral in $M'$, then by Theorem~\ref{thm:nonpreperipheral-hakens-lemma} it is isotopic to a Haken sphere in $M'$. Otherwise, since the complexity of the splitting of $M'$ is smaller than the original one, we can use the inductive hypothesis on $M'$ to conclude that $R'$ is isotopic to a Haken sphere in $M'$. Interpreting $M'$ as a submanifold of $M$, and using that the Heegaard splitting of $M'$ is inherited from $M$, this shows that $R'$ is isotopic to a Haken sphere in $M$ as well.	\end{proof}
	
	Now, if $M$ contains any surviving spheres, then the Surviving Haken Lemma (Theorem~\ref{thm:nonpreperipheral-hakens-lemma}) implies that there is a surviving Haken sphere $\sigma_0$. Connectivity of the surviving sphere complex (Lemma~\ref{lem:connectivity-sphere-graph}), together with the Claim then inductively implies that any surviving sphere is isotopic to a Haken sphere. This finishes the proof of the inductive steps.
\end{proof}

\end{document}